\documentclass[11pt,draft,reqno]{amsart}

\usepackage{amsmath,amssymb,amscd,amsfonts,amsthm}
\usepackage{graphics}
\usepackage{dsfont}
\usepackage[all]{xy}
\usepackage{enumerate}

\usepackage{array,amscd,latexsym, verbatim}

{\catcode`\@=11
\gdef\n@te#1#2{\leavevmode\vadjust{%
 {\setbox\z@\hbox to\z@{\strut#1}%
  \setbox\z@\hbox{\raise\dp\strutbox\box\z@}\ht\z@=\z@\dp\z@=\z@%
  #2\box\z@}}}
\gdef\leftnote#1{\n@te{\hss#1\quad}{}}
\gdef\rightnote#1{\n@te{\quad\kern-\leftskip#1\hss}{\moveright\hsize}}
\gdef\?{\FN@\qumark}
\gdef\qumark{\ifx\next"\DN@"##1"{\leftnote{\rm##1}}\else
 \DN@{\leftnote{\rm??}}\fi{\rm??}\next@}}
\DeclareFontFamily{OT1}{wncyr}{\hyphenchar\font45 }
\DeclareFontShape{OT1}{wncyr}{m}{n}{%
   <5> <6> <7> <8> <9> gen * wncyr
   <10> <10.95> <12> <14.4> <17.28> <20.74>  <24.88>wncyr10}{}
\DeclareFontShape{OT1}{wncyr}{m}{it}{%
   <5> <6> <7> <8> <9> gen * wncyi
   <10> <10.95> <12> <14.4> <17.28> <20.74> <24.88> wncyi10}{}
\DeclareFontShape{OT1}{wncyr}{m}{sc}{%
   <5> <6> <7> <8> <9> <10> <10.95> <12> <14.4>
   <17.28> <20.74> <24.88>wncysc10}{}
\DeclareFontShape{OT1}{wncyr}{b}{n}{%
   <5> <6> <7> <8> <9> gen * wncyb
   <10> <10.95> <12> <14.4> <17.28> <20.74> <24.88>wncyb10}{}
\input cyracc.def
\def\rus{\usefont{OT1}{wncyr}{m}{n}\cyracc\fontsize{9}{11pt}\selectfont}
\def\rusit{\usefont{OT1}{wncyr}{m}{it}\cyracc\fontsize{9}{11pt}\selectfont}

\DeclareMathSizes{9}{9}{7}{5}

\theoremstyle{plain}

\newtheorem{theorem}{Theorem}
\newtheorem*{thmnonumber}{Theorem}
\newtheorem{proposition}{Proposition}
\newtheorem{lemma}{Lemma}
\newtheorem{claim}{Claim}
\newtheorem{corollary}{Corollary}

\theoremstyle{definition}

\newtheorem{remark}{\it Remark}

\newtheorem{nothing*}[theorem]{}
\newtheorem{subnothing*}[sub]{}
\newtheorem{example}{Example}

\newtheorem{question}{Question}

\theoremstyle{remark}

\newcommand{\lb}{{\rm(\hskip -.1mm}}
\newcommand{\rb}{{\hskip .3mm\rm)}}

\def\AX {{\rm Aut}(X)}
\def\ft {\varphi_t}
\def\fat {\{\ft\}_{t\in T}}
\def\fT{\varphi^{\ }_T}
\def\fs {\psi_s}
\def\fas {\{\fs\}_{s\in S}}
\def\fS{\psi^{\ }_S}
\def\XY{X\times^{\ }_{Z}X}

\begin{document}

%%\title[Separating by rational invariants]{Separating by rational invariants}
%%{Rational invariants of\\ infinite dimensional \\ algebraic transformation groups}
%%\title[]{Orbits and rational invariants of infinite dimensional groups of automorphisms}
\title[On infinite dimensional algebraic transformation groups%%groups of automorphisms
]{On infinite dimensional\\
algebraic transformation groups}
%%groups of automorphisms of\\ algebraic varieties}

\author[Vladimir  L. Popov]{Vladimir  L. Popov${}^*$}
\address{Steklov Mathematical Institute,
Russian Academy of Sciences, Gubkina 8, Moscow 119991,
Russia}
\address{National Research University Higher School of Economics, Myasnitskaya
20, Moscow 101000, Russia} \email{popovvl\char`\@mi.ras.ru}

\dedicatory{To E.\;B.\;Dynkin on his $90$th anniversary}%%birthday}

\thanks{
 ${}^*$\,Supported by
 grants {\rus RFFI
11-01-00185-a}, {\rus N{SH}--5139.2012.1}, and the
program {\it Contemporary Problems of Theoretical
Mathematics} of the Russian Academy of Sciences, Branch
of Mathematics.}

\maketitle

\begin{abstract} We explore orbits, rational invariant functions, and quotients
of the natural actions of connected, not necessarily finite dimensional subgroups of the automorphism groups
of irreducible algebraic varieties.\;The applications of the results obtained are given.
\end{abstract}

%%\

%%\vskip 4mm
\subsection*{1. Introduction} The
following
well-known result (see, e.g.,\,{\cite[Prop. I.2.2]{Bor91}})
is one of the
%%basic
indispensable tools
%%instruments
%%of
in the theory of algebraic groups:

\begin{thmnonumber}%%[cf., e.g.,\,{\cite[Prop.\;I.2.4]{Bor91}}]
Let $\varphi_i\colon T_i\to G$ $(i\in \mathcal I)$ be a collection of morphisms from irreducible algebraic varieties $T_i$ into an algebraic group $\,G$, and assume that the identity element of $\,G$ lies in $X_i:=\varphi_i(T_i)$ for each $i\in \mathcal I$.\;Then the subgroup $A$ of $\,G$ generated, as an abstract group, by the set $M:=\bigcup_{i\in \mathcal I} X_i$  coincides with the intersection of all closed subgroups of $\,G$ containing $M$.\;Moreover, $A$ is connected and there is a finite sequence
$(\alpha_1,\ldots, \alpha_n)$ in $\mathcal I$ such that $A=X^{e_1}_{\alpha_1}\cdots X^{e_n}_{\alpha_n}$, where $e_i\!=\!\pm 1$ for each\;$i$.
\end{thmnonumber}

%%\eject

 Here we show that the analogous construction, applied in place of $G$ to $\AX$, where $X$ is an irreducible algebraic variety,
yields
%%, in place of $A$,
a group,
%%which is not, in general, algebraic,
%%,
though not in general algebraic,
but
whose natural action on $X$
surprisingly
 retains some basic properties
 %%\newline
%%
 %%\eject
 \noindent 
 of orbits and
 %%invariant fields
 fields of invariant rational functions for algebraic group actions.\;This leads to some applications.

  In general, the groups $\AX$ are infinite dimensional.\;Endowing them with the structures of infinite dimensional algebraic groups goes back to \cite{Sha66}, \cite{Sha82}, where this is done for $X={\mathbf A}\!^n$ (in modern terminology, the affine Cremona group ${\rm Aut}({\mathbf A}\!^n)$ is the ind-group).\;A modification of the argument from \cite{Sha82} shows that $\AX$ is actually an ind-group for any affine $X$.\;

  In \cite{MO67} a
  functorial approach to $\AX$ was developed.\;

  The important concept of algebraic family in $\AX$ was introduced and ela\-bo\-rated in \cite{Ram64}; this led to the notions of a connected subgroup and an infinite dimensional subgroup of $\AX$.\;Later in \cite{Ser10} the same idea was embodied in the definition of the Zariski topology of $\AX$ and ${\rm Bir}(X)$ (see Remark \ref{idea} below).\;

  In \cite{Ram64} it was for the first time
  %%for the first time
  discovered
  %%that some
 that infinite dimensional connected subgroups of $\AX$ retain some properties of finite dimensional ones, namely, that orbits are open in their closure.\;

 A simple method of constructing many one-dimensional
  unipotent subgroups of $\AX$  by means of a single such subgroup $U$ was described in \cite{Pop87} (in \cite{AFKKZ}
 %% these subgroups
 they are called replicas of $U$) and applied
  %%in \cite{Pop87}
  to constructing
  non-trian\-gu\-lar actions of ${\mathbf G}_a$ on ${\mathbf A}\!^n$.\;This method was then used in the proof of a statement, still existing as folklore
  (see Appendix),
  that ensures
  %%and ensuring
  %% and showing
  infinite dimensionality of $\AX$ in many cases.\;

  In
  %%for many $X$.\;In
  \cite[Defs.\;2.1 and 2.2]{Pop05} attention was drawn to considering
   the subgroups of $\AX$ (in general, infinite dimensional) generated by one-dimensional unipotent subgroups of $\AX$.\;They
   were then applied in \cite{Pop11} to constructing a big stock of varieties with trivial Makar-Limanov invariant.\;Later this topic was
   %%then
   de\-ve\-loped further in \cite{AFKKZ}, where the subgroup
   %%${\rm SAut}(X)$
   of $\AX$ generated by all one-dimensional unipotent subgroups of $\AX$ was considered.\;If
   %%(in \cite{AFKKZ} it is denoted in by ${\rm SAut}(X)$).\;If
   %%there are sufficiently many
   these one-dimensional subgroups
   %%in the sense that
   %%they
   act on $X$ ``in all directions'', in \cite{AFKKZ}, using replicas, the geometric manifestation of infinite dimensionality of this subgroup\,---\,infinite
   %%${\rm SAut}(X)$\,---\,infinite
   transitivity of
   %%the
   its action
   %%of ${\rm SAut}(X)$
   on $X$\,---\,was proved.\;In
  %%
  %%
  %% was proved, using replicas, the geometric manifestation of infinite dimensionality of ${\rm SAut}(X)$\,---\,infinite transitivity of the action of ${\rm SAut}(X)$ on $X$.\;In
  \cite{AFKKZ} another property was found
  %%
  %%was found another property
  retained
   under passing from finite dimensional groups to some infinite dimensional ones:\;it was
   proved that the analogue of classical Rosenlicht's theorem about the existence of rational quotient holds for any subgroup of $\AX$
   gene\-ra\-ted by a collection of finite dimensional connected algebraic subgroups.
   %% (e.g., for
   %%${\rm SAut}(X)$).

   In the present paper we show that actually the analogue of classical Rosenlicht's theorem holds true
   %%in greater generality, namely,
   for {\it every} connected subgroup $G$ of $\AX$.\;The proof is heavily based on another result proved in this paper:\;loosely speaking,
   it claims that  the action of $G$ on $X$
   %%in a sense
   is
   in a sense
   ``reduced'' to the ``action'' of a finite dimensional family in $\AX$.\;The applications of these results concern, in particular, the topic of multiple transitivity of the actions on $X$ of
   connected subgroups of $\AX$; we show that it is intimately related to unirationality of $X$.\;We demonstrate
   how this can be applied
  %% This, in turn, is applied
  to proving unirationality of some varieties, e.g., the Calogero--Moser spaces and
   the varieties of $n$-dimensional representations of a fixed representation type of a finitely generated free associative algebra.\;The precise formulations of the
  %%and to defining some subgroups canonically related to $\AX$.\;The
  main results are given in Section\;3, and the necessary definitions are collected in Section\;2.

\vskip 2mm

In what follows,  variety means algebraic
variety in the sense of Serre over an algebraically closed field $k$ of arbitrary characteristic (so algebraic group means algebraic
group over $k$).\;The standard notation and conventions of
[Bo91] and [PV94] are used freely.\;Given a rational function $f\in k(X)$ and an element $\sigma\in \AX$, we denote by $f^\sigma$ the rational function on $X$ defined by $f^\sigma(\sigma(x))=f(x)$ for every point $x$ in the domain of definition of\;$f$.
%%\;In particular, $k[X]$ and $\AX$ denote resp., the $k$-algebra of regular functions
%%and the automorphism group of a variety $X$, and
%%$k(X)$ denotes the field of rational functions of an irreducible variety
%%$X$.\;${\rm Bir}(X)$ denotes the group of birational self-maps
%%of an irreducible variety $X$.
%%Given the automorphisms $g, h\in \AX$, we denote its composition $g\circ h$ by $gh$.

\vskip 2mm
{\it Acknowledgement.} I am grateful to the referee for thoughtful reading and suggestions.

  \subsection*{2. Definitions and notation}  Let $T$ be  an irreducible variety. Any map
\begin{equation*}
\varphi\colon T\to \AX, \;\;t\mapsto \ft,
\end{equation*}
determines a {\it family}
$\fat$ in $\AX$ parameterized by $T$.
We put
\begin{equation*}
\varphi^{\ }_T:=\varphi(T).
%%\{\ft\mid t\in T\}\subseteq \AX.
\end{equation*}

If $\mathcal I$ is a nonempty collection of families
in $\AX$, then the subgroup of $\AX$ generated, as an abstract group, by the set
$\bigcup\fT$ with the union taken over all families $\fat$ in $\mathcal I$ will be called
the {\it group generated by $\mathcal I$}.

We shall say that a family $\fat$ in $\AX$ is
\begin{enumerate}[\hskip 4.0mm ---%%$\cdot$
]
\item {\it injective} (see\;\cite{Ram64}) if $\ft\neq \varphi_{s}$ for all $t\neq s$;
%% (cf.\;\cite{Ram64});
\item {\it unital} if ${\rm id}_X\in \fT$;
%%, the identity element of $\AX$;
\item {\it algebraic} (see\;\cite{Ram64}) if
\
\vskip -3.7mm
\
\begin{equation}\label{wtilde}
\widetilde\varphi\colon T\times X\to X,\quad  (t,x)\mapsto \ft(x)
\end{equation}
\
\vskip -5mm
\
\noindent
 is a morphism.
 %% (cf.\,\cite{Ram64}).
\end{enumerate}

If $\fat$ is an algebraic family in $\AX$ and $\tau\colon S\to T$ a morphism, then
$\{\psi_s:=\varphi_{\tau(s)}\}_{s\in S}$ is also an algebraic family in $\AX$. If $\tau$ is surjective, then
$\varphi_T=\psi_S$.\;Since $S$ may be taken smooth and $\tau$ surjective
(even such that $\dim  S=\dim T$ and $\tau$ is proper \cite{Jon96}), every subgroup of $\AX$ generated by a collection of unital algebraic families in $\AX$ is also ge\-nerated by a collection of unital algebraic families $\fat$ with smooth\;$T$.

Given a family $\fat$ in $\AX$, the family $\{\varphi^{-1}_t\}^{\ }_{t\in T}$ in $\AX$ will be called
%%we shall call the family $\{\varphi^{-1}_t\}^{\ }_{t\in T}$ in $\AX$
the {\it inverse} of $\fat$.\;If
%%Given
%%a finite sequence
$\{\varphi^{\ }_{t}\}^{\ }_{t\in T},\ldots, \{\psi^{\ }_{s}\}^{\ }_{s\in S}$ is a finite sequence of families in $\AX$,
%%we shall call
the family
\begin{equation}\label{prod}
\{\varphi^{\ }_{t}\circ \cdots\circ \psi^{\ }_{s}\}^{\ }_{(t,\ldots, s)\in T\times\cdots\times S}
\end{equation}
in $X$
will be called the {\it product} of
$\{\varphi^{\ }_{t}\}^{\ }_{t\in T},\ldots, \{\psi^{\ }_{s}\}^{\ }_{s\in S}$.\;The inverses and products of families contained in a subgroup $G$ of $\AX$ are contained in $G$ as well.\;The inverses and products of algebraic (resp.,\,unital) families are algebraic; see \cite{Ram64} (resp.,
unital).

Let  $\mathcal I$ be a collection of families in $\AX$.\;We shall say that a {\it family $\fat$ in $\AX$ is derived from $\,\mathcal I\,$} if $\fat$ is a product of fa\-mi\-lies each of which is either a family from $\mathcal I$
or the inverse of such a family.

%%We shall say that a {\it family $\fat$ is generated by a collection $\mathcal I$ of families in $\AX$} if
%%$\fat$ is a product of families each of which is either a family in $\mathcal I$
%%or the inverse of such a family.
%% and
%%\begin{equation*}\label{pi}
%%rt
%%\end{equation*}
%%,
 %%(see \cite{Ram64}),
%% and the same holds with ``algebraic'' replaced by ``unital''.

%%We shall say that a family $\fat$ in $\AX$ is {\it generated by a collection $\mathcal I$ of families} in $\AX$, %%
%%if $\fat$ is a product of families each of which either lies in $\mathcal I$ or is inverse of a family lying in %%$\mathcal I$.
%%from $\mathcal I$ or their inverses.

A subgroup $G$ of $\AX$ is called (see \cite{Ram64}) a {\it finite dimensional subgroup} if there is an integer $n$ such that $\dim T\leqslant n$ for every injective algebraic family $\fat$ in this subgroup; the smallest $n$ satisfying this property is called the {\it dimension of $G$}. If $\,G$ is not finite dimensional, it is called
%%; the smallest $n$ having this property is called \cite{Ram64} the {\it dimension of this subgroup}.\;If such an $n$ does not exist, $H$ is called
an {\it infinite dimensional subgroup} of $\AX$.

If  for every element $g\in G$ there exists a unital algebraic family $\fat$ in $G$ such that $g\in \fT$, then $\,G$ is  called  (see\;\cite{Ram64})  a {\it connected subgroup of $\AX$}.

If $\fat$ is an algebraic family such that $T$ is a connected algebraic group and $\widetilde \varphi$ (given by \eqref{wtilde}) is
an action of $T$ on $X$, then $\varphi^{\ }_T$ is  a connected finite dimensional subgroup of
%%of automorphisms of $X$
$\AX$.\;By \cite[Thm.]{Ram64}, every
connected finite dimensional subgroup of $\AX$
%%such subgroup
is obtained in this way.\;Such subgroups are called {\it connected algebraic subgroups of $\AX$}.

Given a nonempty subset $S$ of $\AX$, we put
  \begin{equation*}
  S(x):=\{g(x)\mid g\in S\}.
  \end{equation*}
Given a subgroup $G$ of $\AX$  and a $G$-invariant subset $Y$ of $X$, we shall say that a family $\fat$ in $G$ is an {\it exhaustive family for the natural action of $\,G$ on $Y$} if
%%\begin{equation*}
$G(y)=\fT(y)\;\mbox{for every point $y\in Y$}$.
%%\end{equation*}

%%Clearly, if $\fat$ is exhaustive for the natural action of $G$ on $Y$ and $\fT\subseteq \fS$ for a family
%%$\fas$ in $G$, then $\fas$ is also
%%exhaustive for this action.  \mnote{move}
\begin{remark}\label{idea} \cite{Ram64} and this paper demonstrate the fruitfulness of the
idea of considering specific families.\;Another
%%embodiment of this idea
  example of its embodiment is obtained by replacing  $\AX$ by ${\rm Bir}(X)$ and algebraic families  by rational ones (i.e., such that $\widetilde\varphi$ is a rational map):\;e.g., using such families, J.-P. Serre defines the important notion of the Zariski topology on the Cremona groups \cite{Ser10}.\;One can expect fruitfulness of
  its implementation in other categories (holomorphic families, differentiable families, $\ldots$).
  %% may also be fruitful.
\end{remark}

\subsection*{3.\;Main results} In  Lemma \ref{con-equiv}, Theorems \ref{main1}, \ref{main3}, \ref{main2} and Corollaries \ref{c1}, \ref{c2} below,
%%``connected group $G$ of automorphisms of an irreducible variety $X$'' does {\it not} mean that
we do {\it not}
%% impose the constraint of
assume finite dimensionality of $G$.\;If
%%$G$ is necessarily finite dimensional: it may or may not be so.\;If
$G$ is finite dimensional, then the statement of
%%Theorems \ref{con-equiv} and
Theorem \ref{main1} becomes trivial and that of Theorems \ref{main3}, \ref{main2} and Corollaries \ref{c1}, \ref{c2}
 turn into the well-known classical results of the algebraic transformation group theory (see, e.g.,\,\cite[Sect.\,1.4, 2.3]{PV94}); in particular, Theorem \ref{main2} becomes classical Rosenlicht's theorem \cite{Ros56}.
%%\mnote{reference to Mumford?}
%%\cite{Ros56}.\mnote{references}

The proofs of the following statements are given in the next sections.

\begin{lemma}\label{con-equiv} Let $X$ be an irreducible variety and let $G$ be a subgroup of $\AX$. Then the following properties are equi\-valent{\rm:}
\begin{enumerate}[\hskip 2.2mm \rm(i)]
\item $G$ is a connected subgroup of $\AX${\rm ;}
\item $G$ is generated by a collection $\mathcal I$ of unital algebraic families
in $\AX$.
%%
%%
%%there exists a collection $\mathcal I$ of unital algebraic families
%%$\fat$
%%in $\AX$ such that $G$ is generated, as an abstract group, by the set
%%$\bigcup\fT$ with the union taken over all families $\fat$ in $\mathcal I$.
%%the union of
%%as an abstract group by
%%all the sets $\fT$ where $\fat$ runs through  $\mathcal G$.
%%
%%
%%are  unital algebraic families
%%$\{(\varphi_{i})_{t_i}\}_{t_i\in T_i}$, $i\in I$ in $\AX$
%%such that $G$ is generated
%%as abstract group by all the sets
%%$(\varphi_i)_{T_i}$, $i\in I$.
 \end{enumerate}
\end{lemma}

The proof is given in Section 4.
%%If condition (ii) of Theorem \ref{con-equiv} holds, we shall say that the group $G$  {\it is ge\-nerated by
%%the collection of families $\mathcal I$}.
%%of unital families in $\AX$}.

\begin{theorem}\label{main1}  Let $X$ be an irreducible variety and let $G$ be a subgroup of $\AX$ generated by a collection $\mathcal I$ of unital algebraic families in $\AX$. Let $\,Y$ be
a $G$-invariant  locally closed subvariety of $X$.\;Then there is a family
%%$\{\varepsilon_t\}_{t\in E}$
derived from $\mathcal I$ and exhaustive for the natural action of $\,G$ on\;$Y$.
%%\;Moreover, there is such $\{\varepsilon_t\}_{t\in E}$  with smooth $E$.
\end{theorem}

The proof is given in Section 6.

%%\begin{corollary} \label{smooth} Let $X$ be an irreducible variety and let $G$ be a connected %%subgroup of  ${\rm Aut}(X)$.\;Let $\,Y$ be
%%a $G$-invariant  locally closed subvariety of $X$.\;Then there is a unital algebraic family $\fat$ in %%$\AX$
%%exhaustive for the natural action of $\,G$ on\;$Y$ and such that $T$ is smooth..
%%\end{corollary}

%%The proof is given in Section 7.

Orbits of connected subgroups of ${\rm Aut}(X)$ are locally closed subvarieties of $X$ (see below, Lemma \ref{orbit}), so one can speak about their dimension.

\begin{theorem}\label{main3} Let $X$ be an irreducible variety and let $G$ be a connected subgroup of  ${\rm Aut}(X)$.\;Let $\,Y$ be an irreducible
$G$-invariant  locally closed subvariety of $X$.\;Then there exists an integer $m^{\ }_{G, Y}$ and a dense open subset $U$ of $\;Y$\,such that
$\dim G(y)=m^{\ }_{G, Y}$ for every point $y\in U$.
\end{theorem}

The proof is given in Section 7.

\begin{theorem}\label{main2}
Let $X$ be an irreducible variety and let $G$ be a connected subgroup of  ${\rm Aut}(X)$.\;Let $\,Y$ be an irreducible
$G$-invariant  locally closed subvariety of $X$.\;Then for some $G$-invariant dense open subset $U$ of $\,Y$
there exists a geometric quotient, i.e.,  there are an irreducible variety $Z$ and a morphism
$\rho\colon U\to Z$ such that
 \begin{enumerate}[\hskip 4.2mm \rm(i)]
 \item $\rho$ is surjective, open, and the fibers of $\rho$ are the $G$-orbits in $U$;
 \item if $\,V$ is an open subset of $U$, then
 \begin{equation*}\rho^*\colon k[\rho(V)]\to \{f\in k[V]\mid f\;\mbox{\it is constant on the fibers of $\rho|^{\ }_V$}\}
 \end{equation*}
 is an isomorphism of $k$-algebras.
 \end{enumerate}
\end{theorem}

The proof is given in Section 8.

\begin{corollary}\label{c1}
Let $X$ be an irreducible variety and let $G$ be a connected subgroup of  ${\rm Aut}(X)$.\;Let $\,Y$ be an irreducible
$G$-invariant  locally closed subvariety of $X$.\;Then there exists a finite subset of $k(Y)^G$ that separates $G$-orbits of points of a dense open subset of $\,Y$.
\end{corollary}
\begin{corollary}\label{c2}
Let $X$ be an irreducible variety and let $G$ be a connected subgroup of  ${\rm Aut}(X)$.\;Let $\,Y$ be an irreducible
$G$-invariant  locally closed subvariety of $X$.\;Then the transcendence degree of the field $k(Y)^G$ over $k$ is equal to $\dim X-m^{\ }_{G, Y}$\;{\lb}\hskip -.4mm see Theorem {\rm \ref{main3}}\hskip -.6mm\rb.\;In particular, $k(Y)^G=k$ if and only if there is an open $G$-orbit in $Y$.
\end{corollary}

Here are some applications of these results.

\begin{theorem}\label{appl1} Let $X$ be a nonunirational irreducible variety.\;Then there exists a nonconstant rational
function on $X$ which is $G$-invariant for every connected
affine algebraic subgroup $G$ of $\AX$.
%%such that for every connected affine algebraic group $H$ and every regular action $\alpha$ of $H$ on
%%$X$, the function $f$ is $H$-invariant.
\end{theorem}

The proof is given in Section 10.

Theorem \ref{appl1} shows  that there is a certain rigidity for the orbits
of any connected affine algebraic group $G$  acting regularly on an irreducible nonunirational variety $X$:
every such orbit should lie in a level variety of a certain nonconstant
rational function on $X$ not depending on $G$ or
%%and
on its action on\;$X$.

\begin{remark}  ``Nonunirational'' in Theorem \ref{appl1} can not be replaced by ``nonrational''.
Indeed,
 %%or by ``not stably rational''.\;Indeed,
 by \cite[Thm.\;2]{Pop13} there exist a connected linear algebraic group $G$ and its finite subgroup $F$ such that $X:=G/F$ is not stably rational; since the natural action of $G$ on $X$ is transitive, $k(X)^G=k$.
\end{remark}

We shall say that $\AX$ is {\it generically $n$-transitive} if there exists a dense open subset $X_n$ of $X$ such that for every point $x, y\in (X_n)^n$ lying off the union of the ``diagonals'', there exists an element $g\in \AX$ such that $g(x)=y$ for the diagonal action
of $\AX$ on $X^n$.

%%There are many generically
 In the literature there are many examples of generically $n$-transitive varieties with $n\geqslant 2$; see \cite{Rei93}, \cite{KZ99}, \cite{Pop07}, \cite{AFKKZ}, \cite{BEE14}.  Uni\-ra\-tionality of these varieties is proved in many cases (see, e.g., \cite[Prop.\;5.1]{AFKKZ}) and no examples of  nonunirational varieties of this type are known.\;The
 following
 %%next
 Theorem\;\ref{appl2} and Corollary \ref{proje} concern this topic and
 make it more likely that such examples do not exist; in the proof we shall assume that $k$ is uncountable, e.g., $k=\mathbf C$.
%%; a wast class of them, called flexible varieties, is recently constructed  in \cite{AFKKZ}.

\begin{theorem}\label{appl2}  Let $X$ be an irreducible variety such that  $\AX$ is generically $2$-transitive.\;Then at least one of the following holds:
\begin{enumerate}[{\hskip 0.5mm} \rm(i)]
\item $X$ is unirational{\rm;}
\item $\AX$ contains no nontrivial connected
%%affine
algebraic subgroups.
%%; if $k$ is uncountable, then
%%$\AX$ contains no nontrivial connected
%%algebraic subgroups.
\end{enumerate}
%%If, moreover, there is no dominant morphism $Z\to X,$ where $Z$ is an abe\-lian variety
%%{\rm(}e.g., if $X$ is not projective{\rm)}, then {\rm``}affine{\rm''} in {\rm(ii)} may be
%%dropped.
%%replaced by
%%\begin{enumerate}[{\hskip 0.5mm} \rm(i)]
%%\item[\rm(iii)]
%%\end{enumerate}
%%
%%
%%
%%Let $X$ be a nonunirational irreducible variety.
%%\begin{enumerate}[{\hskip 0.5mm} \rm(i)]
%%\item If the group $\AX$ is generically $2$-transitive,
%% %%{\lb}i.e., acts transitively on points of $X\times X$ lying off a proper closed subset\rb,
%%  then $\AX$ contains no nontrivial connected affine algebraic subgroups.
%%    \item If, moreover, there is no dominant morphism $Z\to X,$ where $Z$ is an abe\-lian %%variety,
%%    %%of a product of abelian varieties to $X$,
%%    then $\AX$ contains no nontrivial connected algebraic subgroups.
%%        \end{enumerate}
\end{theorem}

The proof is given in Section 11.

In fact, I have
%%I do not know
no examples of $X$ such that $\AX$ is generically $2$-tran\-sitive and
contains no nontrivial algebraic subgroups.

\begin{corollary}\label{proje}  Let
%%the field $k$ be uncountable and let
$X$ be an irreducible complete variety.\;If  $\AX$ is gene\-ri\-cally $2$-transitive, then $X$ is unirational.
\end{corollary}

The proof is given in Section 12.

As
%%the
applications of Theorem \ref{appl2}, we obtain the following Corollaries \ref{ColMos}  and\
\ref{reptyp}:
%%;
 %%are known to me; %%such that property (ii) holds;
%%property (ii)
%%Theorem \ref{appl2} %%seems to make
%%makes it more likely that they do not exist.
%%it seems unlikely that they exist.

\begin{corollary}\label{ColMos} Every Calogero--Moser space
%%${\mathcal C}_n:=\widetilde{\mathcal C}_n/\!\!/{\rm PGL}_n({\mathbf C})$, where
%%{\lb}varieties\rb
\begin{equation*}
{\mathcal C}_n:=\{(A, B)\in {\rm Mat}_n({\mathbf C})^2\mid {\rm rk}([A, B]+I_n)=1\}/\!\!/{\rm PGL}_n({\mathbf C})
\end{equation*}
{\lb}see {\rm \cite{{Wil98}}}\rb\;is an irreducible unirational variety.
%%\;The variety $\widetilde{\mathcal C}_n$ is irreducible and unirational as well.
\end{corollary}

The proof
%%, based on multiple transitivity of ${\rm Aut}({\mathcal C}_n)$,
is given in Section 13; it is based on Theorem \ref{appl2} and multiple transitivity of ${\rm Aut}({\mathcal C}_n)$.\;Using other special properties of ${\mathcal C}_n$, one
can prove that  ${\mathcal C}_n$ is actually rational; see Remark \ref{ration} in Section 13.

%%\begin{question} Is ${\mathcal C}_n$ rational?
%%\end{question}

\begin{corollary}\label{reptyp} %%Let ${\rm char}\,k=0$.
For $\,{\rm char}\,k=0$ and $m\geqslant 3$, every set
$Q_{m,n}(\tau)$
 of all points of $\;{\rm Mat}_n(k)^m/\!\!/{\rm PGL}_n(k)$
of a
fixed
representation type
$\tau$
{\lb}see {\rm\cite{Rei93}}\rb\;is an irreducible unirational variety.
%%for $\,{\rm char}\,k=0$ and $m\geqslant 3$.
\end{corollary}

The proof
%%, based on multiple transitivity of ${\rm Aut}(Q_{m,n}(\tau))$,
is given in Section 14;  it is based on Theorem \ref{appl2} and the multiple transitivity of ${\rm Aut}(Q_{m,n}(\tau))$.

\begin{question} Is $Q_{m,n}(\tau)$ rational?
\end{question}

Other applications are discussed in Section 10.

%%The proofs of the statements formulated in this section are given below.

\subsection*{4.\;Proof of Lemma \ref{con-equiv}}
(i)$\Rightarrow$(ii): For every element $g\in G$, fix a unital algebraic family $\fat$ in $G$ such that $g\in \fT$; the connectedness of $G$ implies that such a family exists.\;Then
$G$ is generated, as an abstract group, by $\bigcup\fT$ with the union taken over all the fixed families.

(ii)$\Rightarrow$(i): Since the inverse of any family in $G$ is also a family in $G$, we may (and shall) assume that if a family belongs to $\mathcal I$, then its inverse belongs to $\mathcal I$ too.\;Then for every element $g\in G$,
there exists a finite sequence of families $\fat,\ldots, \{\psi_s\}_{s\in S}$ from $\mathcal I$ such that
$
g=\varphi_{t_0}\circ\cdots\circ\psi_{s_0}
$
for some $t_0\in T,\ldots, s_0\in S$.\;Hence $g$ is contained in the product of families
$\fat,\ldots, \{\psi_s\}_{s\in S}$ defined by  \eqref{prod}.\;Therefore, $G$ is connected.
$\quad  \square$

\subsection*{5.\;Algebraic families}
This section contains several general facts
%%that are
utilized in the proofs of Theorems \ref{main1}, \ref{main3}, and \ref{main2}.
%%statements formulated in the previous section.
\begin{lemma}\label{algfam}Let $X$ be an irreducible variety, let $G$ be a connected subgroup of  ${\rm Aut}(X)$, and let $\,Y$ be a
$G$-invariant  locally closed subvariety of $X$.
\begin{enumerate}[\hskip 2.2mm \rm(i)]
\item Every product of unital families in $\AX$ contains each of them.

\item If a family
$\fat$ in $G$ is exhaustive for the natural action of $G$ on $Y$, then every family
$\fas$ in $G$ such that $\fT\subseteq \fS$ is
also exhaustive for this action.

\item If $G$ is generated by a collection $\mathcal I$ of unital algebraic families, then $G$ is the union of
all families derived from $\mathcal I$.
\item
%%If $\,Y$ is an irreducible
%%$G$-stable  locally closed subvariety of $X$, then
$G|^{\ }_{Y}:=\{g|^{\ }_{Y}\mid g\in G\}$ is a connected subgroup of ${\rm Aut}(Y)$.

\item If
$\mathcal F$ is a finite set of algebraic families in $G$, then $G$
 contains a unital algebraic family $\fat$
 such that $\fT\supseteq\fS$ for every $\fas$ in $\mathcal F$.
 %%If all the families from $\mathcal F$ are unital, then  their product
 %%of all
 %%the families in $\mathcal F$
 %%in any order may be taken as such $\fat$.
 %%such a $\fat$ exists among families generated by $\mathcal F$.
 \end{enumerate}
\end{lemma}
\begin{proof} (i) and (ii): This is immediate from the definitions.

(iii): The proof is similar to that of implication (ii)$\Rightarrow$(i) of Lemma \ref{con-equiv}.

(iv): If $\fat$ is a unital algebraic family in $G$ containing an element $g\in G$, then $\{\varphi^{\ }_t|^{\ }_Y\}^{\ }_{t\in T}$ is a unital algebraic family in $G|^{\ }_{Y}$ containing the element
$g|^{\ }_Y\in G|^{\ }_{Y}$. Whence the claim.

 (v): Due to (i), the proof is reduced to the case where $\mathcal F$ consists of a single family $\fas$.\;In this case, take  an element $g\in \fS$.\;Since $G$ is connected, it contains
a unital algebraic family $\{\mu_r\}_{r\in R}$
such that $g^{-1}\in \mu^{\ }_R$.\;The product of $\fas$ and $\{\mu_r\}_{r\in R}$ is then
the sought-for family $\fat$.
 \end{proof}

\begin{lemma}\label{invcomp} Let $X$ be an irreducible variety and
let $Y$ be a locally closed subvariety of $\,X$.\;Let $Y_1,\ldots, Y_n$ be all the irreducible components of $\,Y$.\;If
%%let
%%$G$
$\fat$ is a unital algebraic family in
%%connected subgroup of
${\rm Aut}(X)$ such that $Y$ is $\varphi_t$-invariant for every $t\in T$, then
%%
%%.\;Let $Y$ be a $G$-invariant locally closed subvariety of $\,X$ and let
%%$Y_1,\ldots, Y_n$ be all the irreducible components of $\,Y$.\;Then
every $Y_i$ is $\varphi_t$--invariant for every $t\in T$.
\end{lemma}
\begin{proof}
%%Let $\fat$ be a unital algebraic family in $G$.\;For every $t\in T$, since
%%$X\to X$, $x\mapsto \ft$
%%is an automorphism of $X$,
For every point $t\in T$, since $\ft\in\AX$ and $Y$ is $\ft$-invariant, $\ft$ permutes $Y_1,\ldots, Y_n$.\;Put
$$
T_{ij}:=\{t\in T\mid \ft(Y_i)=Y_j\}.
$$

For
every point $x\in Y_i$ consider the morphism
\begin{equation}\label{psix}
\widetilde{\varphi}^{\ }_x\colon T\to X,\quad t\mapsto {\widetilde \varphi}(t, x)=\varphi^{\ }_t(x)
\end{equation}
(see \eqref{wtilde}).\;Then, for every\;$Y_j$,
\begin{equation}\label{permu}
T_{ij}=\textstyle \bigcap_{x\in Y_i} \widetilde{\varphi}^{-1 }_x(Y_j).
\end{equation}
Since $Y_j$ is closed, \eqref{permu} implies the closedness of $T_{ij}$ in $T$.\;Unitality of $\ft$ implies
$T_{ii}\neq \varnothing$.\;From $T=\bigsqcup_{j=1}^{n}T_{ij}$ and the irreducibility of $T$ we then infer that $T=T_{ii}$ for every $i$,\;i.e., $Y_i$ is $\ft$-invariant for every $i$ and $t$.
%%\;Lemma  \ref{con-equiv} then completes the proof.
 \end{proof}

Lemma \ref{invcomp} and the definition of connected subgroups of ${\rm Aut}(X)$ yield

\begin{corollary}\label{invccc}
 Let $X$, $Y$, and $Y_1,\ldots, Y_n$ be
the same as in Lemma {\rm\ref{invcomp}}.\;If
$Y$ is $G$-invariant for a connected subgroup $G$ of ${\rm Aut}(X)$, then every $Y_i$ is $G$-invariant.
%%an irreducible variety and let $G$ be a connected subgroup of
%%${\rm Aut}(X)$.\;Let $Y$ be a $G$-invariant locally closed subvariety of $X$ and let $Y_1,\ldots, Y_n$ be all %%the irreducible components of $Y$.\;Then every $Y_i$ is $G$-invariant.
\end{corollary}

\begin{lemma}\label{orbit} Let $X$ be an irreducible variety and let $G$ be a connected subgroup of  ${\rm Aut}(X)$.\;If $\fat$ is an algebraic family in $G$, and $x$ is a point of $X$, then
\begin{enumerate}[\hskip 2.2mm \rm (i)]
\item $G(x)$ is an irreducible locally closed nonsingular subvariety of $X${\rm;}
\item $\fT(x)$ is a constructible subset of $\,G(x)$.
\end{enumerate}
\end{lemma}
\begin{proof}
(i): This is proved in \cite[Lemma\;2]{Ram64}.

(ii): This follows from the definition of algebraic family and
Chevalley's theorem on the image of morphism.
 \end{proof}

\begin{corollary} Let $X$ be an irreducible variety and let $G$ be a connected subgroup of  ${\rm Aut}(X)$.\;Then $k(X)^G$ is algebraically closed in $k(X)$.
\end{corollary}
\begin{proof} Let $f\in k(X)$ be a root of $t^n+f_1t^{n-1}+\cdots+f_n\in k(X)^G[t]$ and let $a\in X$ be a point where $f$ and every $f_i$ are defined.\;Then by Lemma \ref{orbit}(i) the restriction of $f$ to the irreducible variety  $G(a)$ is a well-defined  rational function $f|^{\ }_{G(a)}\in k(G(a))$.\;The image of the rational map
$f|^{\ }_{G(a)}: G(a)\dashrightarrow k$
%%Then $f(G(a))$
is a finite set
%%subset of $k$
since it lies in
the set of roots of $t^n+f_1(a)t^{n-1}+\cdots+f_n(a)\in k[t]$.\;Irreducibility of $G(a)$ then implies that
this
%%subset
image is a single element of $k$, i.e., $f|^{\ }_{G(a)}$ is a constant.\;Whence
%%This means that
$f\in k(X)^G$.
 \end{proof}

\begin{lemma}\label{generic} Let $X$ be an irreducible variety and let $G$ be a connected subgroup of  ${\rm Aut}(X)$.\;Let $Y$ be a $G$-invariant locally closed subvariety of $\,X$.\;Let
$\fat$ be an
 %%unital
 algebraic family in $G$ such that $\fT(y)$ is dense in $G(y)$ for every point $y\in Y$.\;Then
the product
%%$\fas$
of the inverse of $\fat$  and $\fat$ is the  unital algebraic family $\fas$ in $G$ exhaustive for the natural action of $\,G$ on $Y$.
\end{lemma}
\begin{proof} By the definition of $\fas$,
%%$S=T\times T$ and
\begin{equation}\label{psps}
\psi^{\ }_s=\varphi^{-1}_{t_1}\circ\varphi^{\ }_{t_2}\quad\mbox{for}\quad s=(t_1, t_2)\in S=T\times T.
\end{equation}

Take any points $y_1, y_2\in Y$ such that $G(y_1)=G(y_2)$.\;The density assumption then yields the equality
$\overline{\fT(y_1)}=\overline{\fT(y_2)}$,
where bar stands for the closure in $X$.\;By Lemma \ref{orbit}, this implies
$$\fT(y_1)\cap\fT(y_2)\neq\varnothing;$$
whence, $\varphi^{\ }_{t_1}(y_2)=\varphi^{\ }_{t_2}(y_1)$ for some $t_1, t_2\in T$.\;Therefore, $\psi_{s}(y_1)=y_2$ for $\psi_{s}$ defined by \eqref{psps}.\;Hence
%%$s=(t_1, t_2)\in S=T\times T$ and $\psi_s=\varphi^{-1}_{t_2}\circ\varphi^{\ }_{t_1}$.\;Hence
$\fS(y_1)=G(y_1)$ for every point $y_1\in Y$, i.e., $\fas$ is exhaustive for the action of $G$ on $Y$.\;Its unitality
%%of  $\fas$
follows from \eqref{psps}.
%%, as claimed.
 \end{proof}

%%\begin{proof}[Proof of Theorem {\rm\ref{main1}}]
\subsection*{6.\;Proof of Theorem \ref{main1}} First,  we shall show that it suffices to prove the following ``ge\-ne\-ric'' version of Theorem \ref{main1}:
\vskip 2mm
\noindent{\bf Theorem \boldmath$1^*$.}\;\;{\it   Let $X$, $G$, $\mathcal I$, and $Y$ are the same as in Theorem
{\rm \ref{main1}} and let $Y$ be irreducible.\;Then
%%Let $X$ be an irreducible variety and let $G$ be a subgroup of $\AX$ generated by a collection $\mathcal I$ of unital algebraic families in $\AX$. Let $\,Y$ be an irreducible
%%$G$-invariant  locally closed subvariety of $X$.\;
there exist
a  dense open $G$-invariant subset $U$ in $Y$ and
a unital algebraic family $\fat$ in $G$
such that}
 \begin{enumerate}[\hskip 4.2mm \rm(i)]
 \item {\it $\fat$ is
derived from $\mathcal I$}{\rm;}
 \item {\it $\fT(y)$ is dense in $G(y)$ for every point $y\in U$.}
 \end{enumerate}
%%
%%
%%\vskip 6mm
%%
%%
%%
%%{\it If $G$ is connected, then there  are
%%\begin{enumerate}[\hskip 4.2mm \rm(i)]
%%\item a  $G$-invariant open dense subset $U$ in $X$,
%%\item a unital algebraic family $\fat$ in $G$
%%\end{enumerate}
%% such that $\fT(x)$ is dense in $G(x)$ for every point $x\in U$.}
\vskip 2mm

Indeed, assuming that Theorem $1^*$ is proved, we can complete the proof of Theorem \ref{main1}
%%by induction on $\dim Y$
as follows.

The group $G$ is connected by Lemma \ref{con-equiv}.\;Therefore, every irreducible component of $\,Y$ is $G$-invariant by Corollary \ref{invccc}.\;From this and Lemma \ref{algfam}(i),(ii) we infer that it is sufficient to prove Theorem \ref{main1} for irreducible $Y$.\;%%Assuming that $Y$ is irreducible,
In this case we argue by induction on $\dim Y$.

Namely, the case $\dim Y=0$ is clear.\;Assume that
the claim of Theorem \ref{main1} holds for irreducible $G$-invariant subvarieties in $X$ of dimension $<\dim Y$ and consider the set $U$ from Theorem $1^*$.\;Let $Z_1,\ldots, Z_n$ be all the irreducible components of the variety $Y\setminus U$.\;By Corollary  \ref{invccc},
every $Z_i$ is $G$-invariant.\;Since $\dim Z_i<\dim Y$,  the inductive assumption implies  for every $i=1,\ldots, n$  the existence of a unital algebraic family $\{\psi^{(i)}_{s_i}\}_{s_i\in S_i}$ in $G$ such that
\begin{enumerate}[\hskip 4.2mm \rm(a)]
\item $\{\psi^{(i)}_{s_i}\}_{s_i\in S_i}$ is derived from
$\mathcal I${\rm;}
 \item $\{\psi^{(i)}_{s_i}\}_{s_i\in S_i}$ is exhaustive
for the natural action of $\,G$ on\;$Z_i$.
\end{enumerate}

On the other hand, Theorem $1^*$ and Lemma \ref{generic} imply the existence of a unital algebraic family
$\{\lambda_{r}\}_{r\in R}$ in $G$ such that
\begin{enumerate}[\hskip 4.2mm \rm(c)]
\item[\rm(c)] $\{\lambda_{r}\}_{r\in R}$ is derived from
$\mathcal I${\rm;}
 \item[\rm(d)] $\{\lambda_{r}\}_{r\in R}$ is exhaustive for the natural action of $\,G$ on\;$U$.
 \end{enumerate}

The claim of Theorem \ref{main1}
%%(on the existence of $\{\varepsilon_t\}_{t\in E}$)
now follows from (a), (b), (c), (d) and Lemma \ref{algfam}(i),(ii).
This
%%\;Let $E_{\rm sm}$ be the smooth locus of $E$.\;Then $\{\varepsilon_t\}_{t\in E_{\rm sm}}$ is an %%algebraic family in $G$
completes the proof of Theorem \ref{main1}
%%under the
assuming
%%ption
that Theorem $1^*$ is proved.\quad  $\square$

%%\vskip 6mm
%%
%% $G$ contains
%%an exhaustive algebraic family $\fat$ for $U$  by Lemma \ref{generic}.\;Then any algebraic family %%$\{\gamma_r\}_{r\in R}$ in $G$ such that $\gamma^{\ }_R$ contains $\fT$ and $\psi^{(i)}_{S_i}$ for %%every $i$ is exhaustive for $X$.\;The existence of such $\{\gamma_r\}_{r\in R}$ follows from  Lemma %%\ref{algfam}(iii).\;This completes the proof of Theorem \ref{main1} provided Theorem $2^*$ is proved.

\vskip 2mm

We now turn to the proof of Theorem $1^*$.\;Consider the map
\begin{equation}\label{tau}
\tau^{\ }_Y
%%:=
%%\alpha\times {\rm id}^{\ }_X
\colon G\times Y\to Y\times Y, \quad(g, y)\mapsto (g(y), y).
\end{equation}
Its image $\Gamma^{\ }_Y$ is the graph  of the natural action of $G$ on $Y$:
%%$\alpha$:
\begin{equation}\label{graph}
\Gamma^{\ }_Y=\{(y_1, y_2)\in Y\times Y\mid G(y_1)=G(y_2)\}.
\end{equation}

\begin{claim}\label{claim1} Maintain the above notation.
\begin{enumerate}[\hskip 2.2mm\rm(i)]
\item There exists a
%%unital algebraic
family $\fat$
derived from $\mathcal I$ such that $\tau^{\ }_Y(\fT\times Y)$
contains a dense open subset\;\,$V$\,of $\;\overline{\Gamma}^{\ }_Y$, where
bar stands for
the closure
%%of $\,\Gamma^{\ }_Y$
in $Y\times Y$.
\item $\;\overline{\Gamma}^{\ }_Y$ is irreducible.
%%and $\tau(\fT\times X)$ contains a dense open subset of $\overline{\Gamma}$.
\end{enumerate}
\end{claim}

\begin{proof}[Proof of Claim {\rm \ref{claim1}}.]  If $\fas$ is an algebraic family in $G$, then the subset $\tau^{\ }_Y(\fS\times Y)$ of $\Gamma^{\ }_Y$ is the image of the morphism
%%$%%\widetilde{\psi}\times {\rm id}_Y\colon
\begin{equation*}
S\times Y\to Y\times Y,\quad (s, y)\mapsto (\psi^{\ }_s(y), y)
\end{equation*}
 of irreducible varieties (see \eqref{wtilde}).\;Chevalley's
 theorem on the image of morphism then implies that $\overline{\tau^{\ }_Y(\fS\times Y)}$ is an irreducible subvariety of $\overline{\Gamma}^{\ }_Y$  and $\tau^{\ }_Y(\fS\times Y)$ contains a dense open subset of $\overline{\tau^{\ }_Y(\fS\times Y)}$.

 From $\dim \overline{\Gamma}^{\ }_Y\geqslant \dim \overline{\tau^{\ }_Y(\fS\times Y)}$ we conclude that there exists a fa\-mily $\fat$ derived from $\mathcal I$ on which the maximum of $\dim \overline{\tau^{\ }_Y(\fS\!\times Y)}$ is attained when
$\fas$ runs over all families derived from $\mathcal I$.\;If  $\fas$ is a family derived from $\mathcal I$ such that
$\fT\subseteq\fS$, then the maximality condition and irreducibility of $\overline{\tau^{\ }_Y(\fS\times Y)}$  imply that
\begin{equation}\label{eqv}
\overline{\tau^{\ }_Y(\fS\times Y)}=\overline{\tau^{\ }_Y(\fT\times Y)}.
\end{equation}
%%By \eqref{eqv} and Lemma \ref{algfam}(iii), we may (and shall) assume that $\fat$ is unital.\;

Take an element $g\in G$.\;By Lemma \ref{algfam}(iii),(i), there is an algebraic family $\fas$ in $G$ such that $\fT\subseteq\fS$ and $g\in \fS$.\;From \eqref{eqv} and \eqref{tau} we then conclude that $\Gamma^{\ }_Y\subseteq \overline{\tau^{\ }_Y(\fT\times Y)}$.\;Since $\overline{\tau^{\ }_Y(\fT\times Y)}\subseteq \overline{\Gamma}$, we get
$\overline{\tau^{\ }_Y(\fT\times Y)}=\overline{\Gamma}^{\ }_Y$.
This completes the proof.
%%(ii): This follows from (i) and Chevalley's theorem because the subset $\tau(\fT\times X)$ of $\Gamma$ is the %%image of the morphism
%%$\widetilde{\varphi}\times {\rm id}_X\colon T\times X\to X\times X$ of irreducible varieties.
 \end{proof}

Endow $X\times X$ with  the action of $G$ via the second factor:
\begin{equation}\label{cdotact}
g\cdot (x_1, x_2):=(x_1, g(x_2)), \quad x_i\in X, g\in G.
\end{equation}
The second projection
$X\times X\to X$, $(x_1, x_2)\mapsto x_2$
is then $G$-equivariant and, by \eqref{graph}, $\Gamma^{\ }_Y$ and $\overline{\Gamma}^{\ }_Y$ are $G$-invariant.

\begin{claim}\label{claim2} $\fat$ and\;\,$V$\,in Claim {\rm \ref{claim1}} can be chosen so
that\;\,$V$\,is $G$-invariant.
\end{claim}
\begin{proof}[Proof of Claim {\rm \ref{claim2}}.]
Maintain the notation of Claim \ref{claim1} and consider in $\overline{\Gamma}^{\ }_Y$ the $G$-invariant dense open subset
\begin{equation}\label{covering}
V_0:=\textstyle\bigcup^{\ }_{g\in G} g\cdot V.
\end{equation}
Since $V_0$ is quasi-compact, its covering \eqref{covering} by open subsets $g\cdot V$, $g\in G$, contains a finite subcovering:
\begin{equation}\label{fcovering}
V_0=\textstyle\bigcup_{i}^ng_i\cdot V\quad\mbox{for some elements $g_1,\ldots, g_n\in G$}.
\end{equation}

By Lemma \ref{algfam}(iii), every $g_i$ is contained in a family derived from $\mathcal I$.\;Taking a product of $\fat$ with these families, we obtain a family $\fas$ derived from $\mathcal I$ such that
\begin{equation}\label{includ}
{\fT}\circ g^{-1}_i\subseteq \fS\;\,\mbox{for every $i=1,\ldots, n$.}
\end{equation}

Since $V\subseteq \tau^{\ }_Y(\fT\times Y)$, from \eqref{tau} and \eqref{cdotact} we obtain
\begin{equation}\label{giV}
g_i\cdot V\subseteq \{(\ft(y), g_i(y))\mid t\in T, y\in Y\}.
\end{equation}
This yields
\begin{equation}\label{ii}
\begin{split}
\tau^{\ }_Y(\fS\times Y)&=\big\{\big(\fs(y), y\big)\mid s\in S, y\in Y\big\}\\
&=\big\{\big(\fs(g_i(y)), g_i(y)\big)\mid s\in S, y\in Y\big\}\\[-1.5pt]
&
\supseteq \big\{\big({\ft}\big(g^{-1}_i(g_i(y))\big), g_i(y)\big)\mid t\in T, y\in Y\big\}\quad\mbox{(by \eqref{includ})}
\\[-1.5pt]
&\supseteq g_i\cdot V \quad\mbox{(by \eqref{giV})}.
\end{split}
\end{equation}

Thus $V_0\subseteq \tau^{\ }_Y(\fS\times Y)$
%%From
 by \eqref{fcovering} and \eqref{ii}.\;So,
  %%we infer that $V_0\subseteq \tau^{\ }_Y(\fS\times Y)$.\;So
  replacing $\fat$ and $V$ by, resp.,  $\fas$ and $V_0$,  we may attain that $V$ in Claim \ref{claim1} is $G$-invariant.
 \end{proof}

To complete the proof of Theorem $1^*$,
%%\ref{main1},
consider the second projection
\begin{equation}\label{pi2}
\pi^{\ }_Y\colon \overline{\Gamma}^{\ }_Y\to Y, \quad (y_1, y_2)\mapsto y_2;
\end{equation}
it is a $G$-equivariant surjective morphism of irreducible varieties.\;Let $\fat$ and $V$ be as in Claim \ref{claim1} and let $V$  be $G$-invariant by Claim \ref{claim2}.\;Since $V$ is a dense open subset of
$\overline{\Gamma}^{\ }_Y$, by
Chevalley's
 theorem on the image of morphism, $\pi^{\ }_Y(V)$ contains a dense open subset of $Y$.\;Let $U$ be the union of all dense open subsets of $Y$ lying in  $\pi^{\ }_Y(V)$.\;Since  $V$ is $G$-invariant and $\pi^{\ }_Y$ is $G$-equivariant, $\pi^{\ }_Y(V)$ is $G$-invariant.\;Therefore,
$U$ is
also $G$-invariant.

Take a point $y\in U$.\;Since  $V\subseteq \Gamma^{\ }_Y$, $\pi^{-1}_Y(y)\cap \Gamma^{\ }_Y=\{(g(y), y)\mid g\in G\}$, and $V\supseteq \big\{\big(g(y), y\big)\mid g\in \fT\big\}$,
we have
\begin{align}
\varnothing\neq V\cap \pi^{-1}_Y(y)&=V\cap \Gamma^{\ }_Y\cap \pi^{-1}_Y(y)=V\cap \big\{\big(g(y), y\big)\mid g\in G\big\}\label{end1}\\
&\subseteq \big\{\big(g(y), y\big)\mid g\in \fT\big\}.\label{end2}
\end{align}
By Lemma \ref{orbit}, $\big\{\big(g(y), y\big)\mid g\in G\big\}$ is an irreducible locally closed subset of $\overline{\Gamma}^{\ }_Y$.\;From \eqref{end1} we then infer that  $V\cap \big\{\big(g(y), y\big)\mid g\in G\big\}$ is a dense open subset of  $\big\{\big(g(y), y\big)\mid g\in G\big\}$, and from \eqref{end2} that
$\fT(y)$ is dense in $G(y)$.\;This completes the proof of Theorem $1^*$ and hence that of Theorem\;\ref{main1}.$\quad \square$
%%Since  $V$ is $G$-invariant and $\pi_2$ is $G$-equivariant, $\pi_2(V)$ is $G$-invariant.

%%\subsection*{7.\;Proof of Corollary \ref{smooth}} By Theorem \ref{main1} there is a unital algebraic family $\fas$ in $G$ exhaustive for the action of $G$ on $Y$.\;Since the smooth locus $S_{\rm sm}$ of $S$ is dense in $S$, the algebraic family  $\{\psi_s\}_{s\in S_{\rm sm}}$ in $G$ has the property that $\psi_{S_{\rm sm}}(y)$ is dense in $G(y)$ for every point $y\in Y$.\;Hence by Lemma \ref{generic} the product of the inverse of  $\{\psi_s\}_{s\in S_{\rm sm}}$ and $\{\psi_s\}_{s\in S_{\rm sm}}$ is the sought-for exhaustive family.
%%unital
%%algebraic family in $G$
%%\hfill$\square$

\subsection*{7.\;Proof of Theorem \ref{main3}}
Maintain the notation of the proof of Theorem\;\ref{main1}.\;It is proved there
%%It is proved there
%%There is proved
that the restriction of $\pi^{\ }_Y$ to $V$ is a dominant morphism of irreducible varieties $V\to Y$ whose fiber over every point $y$ of a dense open subset $U$ of $Y$ is isomorphic to a dense subvariety of $G(y)$.\;Hence, the dimension of this fiber is $\dim G(y)$.\;The claim now follows from
the fiber dimension theorem \cite[5.6]{Gro65}.\quad $\square$

\subsection*{8.\;Proof of Theorem \ref{main2}} By Lemma \ref{algfam}(iv), it suffices to give a proof for $Y=X.$\;We shall use the idea utilized in \cite[4]{Lun73} for proving the existence of a generic stabilizer for reductive group actions on smooth affine varieties. Below is maintained the notation used in the proof of Theorem\;\ref{main1}.

The plan is to repeat several times the procedure of replacing $X$ by its open dense subset having some necessary additional properties; in order to avoid unnecessary extra notation,  this subset will still be denoted by $X$.\;An open subset of the original $X$ obtained at the last step will be the sought-for $U$ from the formulation of Theorem \ref{main2}.

Since any subfield of $k(X)$ containing $k$ is finitely generated over $k$, replacing $X$ by an appropriate  invariant dense open subset of $X$ we can (and shall)
find an irreducible affine normal variety $Z$ and a surjective morphism
\begin{equation*}\label{pim}
\rho\colon X\to Z
\end{equation*}
such that
$\rho^*(k(Z))=k(X)^G$.\;This equality implies that $\rho$ is a separable morphism; see, e.g., \cite[AG, Prop.\,2.4]{Bor91}.
%%\end{equation}
%%and
%%\begin{enumerate}[\hskip 5.6mm \rm(a)]
%%\item[(q${}^{\ }_2$)] $Y$ is affine;
%%\item[(q${}^{\ }_3$)] %%$X$ and
%%$Y$ is normal.
%%\end{enumerate}

The construction yields that

%%\eject
\begin{enumerate}[\hskip 5.6mm \rm(a)]
\item[(q${}^{\ }_1$)] $G(x)\subseteq \rho^{-1}(\rho(x))$ for every point $x\in X$.
\end{enumerate}

 By
%%?? \mnote{???}
the fibre dimension theorem and Theorem \ref{main3}, further
replacing $X$ and $Z$ by the appropriate open sets, we can (and shall) attain  the following properties:
%%\mnote{reference,\\ Chevalley?}
\begin{enumerate}[\hskip 5.6mm \rm(a)]
\item[(q${}^{\ }_2$)] for every point $z\in Z$, the dimension of every irreducible component of  $\rho^{-1}(z)$
is equal to $\dim X-\dim Z$;
\item[(q${}^{\ }_3$)] $\dim G(x)=\dim G(x')$ for all
%%every
points $x, x'\in X$.
\end{enumerate}

Lemma \ref{orbit}(i) and (q${}^{\ }_3$) imply that $G(x)$ is closed in $X$ for every point $x\in X$.

By Grothendieck's generic freeness lemma \cite[6.9.2]{Gro65}, after replacing $Z$ by a principal open subset, we can (and shall) assume that
\begin{enumerate}[\hskip 5.6mm \rm(a)]
\item[(q${}^{\ }_4$)] there exists an affine open subset $X_0$ of $X$ such that $\rho(X_0)=Z$ and
    $k[X_0]$ is a free $\rho^*(k[Z])$-module.
\end{enumerate}
Below, for any subsets $S\subseteq X$ and $R\subseteq X\times X$, we put
\begin{equation*}\label{S}
S_0:=S\cap X_0,\quad R_0:=R\cap (X_0\times X_0).
\end{equation*}

Finally, replacing $X$ by the invariant open set $\bigcup_{g\in G} g(X_0)$, we can (and shall) assume that
\begin{enumerate}[\hskip 5.6mm \rm(a)]
\item[(q${}^{\ }_5$)] the intersection of $X_0$ with every $G$-orbit in $X$ is nonempty.
\end{enumerate}

Consider now in $X\times X$ the $G$-invariant (with respect to action \eqref{cdotact}) closed subset
\begin{equation}\label{XYX}
\XY:=\{(x_1, x_2)\in X\times X\mid \rho(x_1)=\rho(x_2)\}
\end{equation}
and its affine open subset $(\XY)_0$.
\begin{claim}\label{denseXY0}
%%$X_0\times_YX_0$
$(\XY)_0$ is dense in $\XY$.
\end{claim}
\begin{proof}[Proof of Claim {\rm\ref{denseXY0}}]
Take a point $(x_1, x_2)\!\in\! \XY$.\;From \eqref{XYX} and (q${}^{\ }_1$) we infer that $G(x_1)\!\times\! G(x_2)\!\subseteq\! \XY$, and from (q${}^{\ }_5$)
and Lemma \ref{orbit}(i) that
$(G(x_1)\!\times\! G(x_2))_0$
%%$G(x_1)_0\!\times\! G(x_2)_0$
is
a dense open subset of $G(x_1)\!\times\! G(x_2)$.\;Therefore, since $(x_1, x_2)\!\in\! G(x_1)\!\times\! G(x_2)$, the closure of
%%$G(x_1)_0\!\times\! G(x_2)_0$
$(G(x_1)\!\times\! G(x_2))_0$ in
$\XY$ contains $(x_1, x_2)$.\;Whence the claim, because
$
%%G(x_1)_0\!\times\! G(x_2)_0
(G(x_1)\!\times\! G(x_2))_0\!\subseteq\!(\XY)_0$.
%%X_0\!\times_Y\! X_0$.
 \end{proof}

Next, consider the set
\begin{equation}\label{Gamma}
\Gamma:=\Gamma^{\ }_X
 \end{equation}
 defined by \eqref{graph}.\;By (q${}^{\ }_1$), we have   $\Gamma\subseteq\XY$. Since $\XY$ is closed in $X\times X$, this yields
$\;\overline{\Gamma}\subseteq\XY$ (see Claim \ref{claim1}(i)).

\begin{claim}\label{fo}
%%If
$\;\overline{\Gamma}=\XY$.
%%, then the claim of Theorem {\rm\ref{main2}} holds.
\end{claim}

First, we shall show how to deduce Theorem \ref{main2} from Claim \ref{fo}.

By \eqref{Gamma} and Claims \ref{claim1}(ii), \ref{fo},  the variety $\;\overline{\Gamma}=\XY$
is irreducible.\;Consi\-der its dense open subset $V$ from Claim \ref{claim2} and
morphism $\pi^{\ }_{X}\colon\;\overline{\Gamma}\to X$ defined by
\eqref{pi2} for $Y=X$.\;If $B$ is an irreducible component of $\overline{\Gamma}\setminus V$ such that
$\pi^{\ }_{X}(B)$ is dense in $X$, then, by the fiber dimension theorem,
$\dim \pi^{-1}_{X}(x)>\dim \pi^{-1}_{X}(x)\cap B$ for every point $x\in X$ lying off a proper closed subset of $X$.\;This and property (q${}^{\ }_3$) imply that $V\cap \pi_{X}^{-1}(x)$ is dense in $ \pi_{X}^{-1}(x)$ for  every such $x$.\;On the other hand,
 $\pi_{X}^{-1}(x)=\rho^{-1}(\rho(x))\times x$ by \eqref{XYX} and, as explained at the end of the proof of Theorem \ref{main1},  $V\cap \pi_X^{-1}(x)$ is a dense open subset of $G(x)\times x$.\;Since $G(x)\subseteq \rho^{-1}(\rho(x))$, this shows that $G(x)$ is dense in $ \rho^{-1}(\rho(x))$.\;The closedness of $G(x)$ in $X$ then
 implies that $G(x)=\rho^{-1}(\rho(x))$ for every point $x\in X$ lying off a proper closed subset.\;This means that
 replacing
 $Z$ by its open subset and $X$ by the inverse image of this subset, we can (and shall) assume that $\rho$ is an orbit map, i.e., the fibers of $\rho$ are the $G$-orbits in $X$.\;Since $\rho$  is a surjective separable morphism and $Z$ is a normal variety,  by \cite[Prop.\,II.6.6]{Bor91} this implies that $\rho\colon X\to Z$ is the geometric quotient.\;Thus
the proof of  Theorem \ref{main2} is completed provided that Claim \ref{fo} is proved.\quad  $\square$

\medskip

So it remains to prove Claim \ref{fo}.

\begin{proof}[Proof of Claim {\rm\ref{fo}}] We divide it into three steps.

  1. In view of Claim \ref{denseXY0}, it suffices to prove the density of
$\Gamma_0$  in  $(\XY)_0$.\;Since
%%Seeing that
$(\XY)_0$ is an affine variety, the latter is reduced to proving that if a function $f\in k[(\XY)_0]$ vanishes on $\Gamma_0$,
\begin{equation}\label{vanishG}
f|^{\ }_{\Gamma_0}=0,
\end{equation}
then $f=0$.\;To prove this, note that the closedness of $(\XY)_0$ in $X_0\times X_0$ implies the existence of a function
$h\in k[X_0\times X_0]$ such that
\begin{equation}\label{h}
h|^{\ }_{(\XY)_0}=f.
\end{equation}
In turn, since $k[X_0\times X_0]=p_1^*(k[X_0])\otimes_k p_2^*(k[X_0])$, where $p_i\colon X_0\times X_0\to X_0$, $(x_1, x_2)\mapsto x_i$, there are functions $s_1,\ldots, s_m$, $t_1,\ldots, t_m\in k[X_0]$
such that
\begin{equation}\label{st}
h=\textstyle \sum_{i=1}^m p_1^*(s_i) p_2^*(t_i).
\end{equation}

%%\begin{claim}\label{li}

2. By an appropriate replacement of $h$ and $s_1,\ldots, s_m, t_1,\ldots, t_m$ we may
%%attain
obtain that $t_1,\ldots, t_m$
are linearly
independent over $\rho^*(k[Z])$.\;Indeed, by  property (q${}^{\ }_4$), there are
functions $b_1,\ldots, b_r\in k[X_0]$,  linearly independent over  $\rho^*(k[Z])$,
 such that
\begin{equation}\label{cij}
t_i=\textstyle \sum_{j=1}^r c_{ij} b_j \quad\mbox{for some}\;\; c_{ij}\in \rho^*(k[Z]),\; i=1,\ldots, m.
\end{equation}
In view of \eqref{st} and \eqref{cij},
we have
\begin{equation}\label{hh}
h=\textstyle \sum_{j=1}^{r}\bigl(\sum_{i=1}^{m} p_1^*(s_i) p_2^*(c_{ij})\bigr) p_2^*(b_j).
\end{equation}

Take a point  $x=(x_1, x_2)\in (\XY)_0$. Since $\rho(x_1)=\rho(x_2)$, we have
\begin{equation}\label{==}
c_{ij}(x_1)=c_{ij}(x_2)\quad\mbox{for all $i, j$.}
\end{equation}
From \eqref{hh} and \eqref{==} we then obtain
\begin{equation}\label{=h=}
\begin{split}
h(x)&=\textstyle \sum_{j=1}^{r}\bigl(\sum_{i=1}^{m} s_i(x_1) c_{ij}(x_2)\bigr) b_j(x_2)\\
&=\textstyle \sum_{j=1}^{r}\bigl(\sum_{i=1}^{m} s_i(x_1) c_{ij}(x_1)\bigr) b_j(x_2).
\end{split}
\end{equation}
Hence if we put
\begin{equation}\label{tilde}
\begin{split}
\textstyle d_j&:=\textstyle\sum_{i=1}^m s_i c_{ij}\in k[X_0],\\
\widetilde h&:=\textstyle \sum_{j=1}^{r}p_1^*(d_j) p_2^*(b_j)\in k[X_0\times X_0],
\end{split}
\end{equation}
then we have $h(x)={\widetilde h}(x)$ by virtue of  \eqref{=h=}.\;Given
 %%and \eqref{tilde}.
 \eqref{h}, this yields
\begin{equation}\label{tldh}
{\widetilde h}|^{\ }_{(\XY)_0}=f.
\end{equation}
From \eqref{tilde} and \eqref{tldh} we conclude that the replacement of $s_1,\ldots, s_m$ and $t_1,\ldots, t_m$ by, respectively, $d_1,\ldots, d_r$ and $b_1,\ldots, b_r$ is the one
%%that
we are looking for.

\smallskip

3. Thus, %%in order not to change
keeping the notation, we shall now assume that $t_1,\ldots, t_m$ in \eqref{st}
are linearly
independent over $\rho^*(k[Z])$.

Take an element $g\in G$ and let $D$ be the domain of definition of the rational function
\begin{equation*}
\ell=\textstyle\sum_{i=1}^m s_it_i^g\in k(X).
\end{equation*}
Since $X$ is irreducible, $D\cap g(D)\cap X_0\cap g(X_0)$ is a dense open subset of $X$.\;Let $x$ be a point of this subset.\;Then the rational functions $\ell$, $s_i$, $t_i^g\in k(X)$ are defined at $x$ and
\begin{equation}\label{Ga0}
a:=(x, g^{-1}(x))\in \Gamma_0.
\end{equation}
From this we obtain
\begin{equation*}
\begin{split}
\ell(x)&=\textstyle \sum_{i=1}^m s_i(x)t_i^g(x)= \sum_{i=1}^m s_i(x)t_i(g^{-1}(x))\\
&\overset{\mbox{\tiny by}\,\eqref{Ga0}}{=\hskip -1mm=\hskip -1mm=\hskip -1mm=}\textstyle \bigl(\sum_{i=1}^m p_1^*(s_i) p_2^*(t_i)\bigr)(a)\overset{\mbox{\tiny by}\,\eqref{st}}{=\hskip -1mm=\hskip -1mm=\hskip -1mm=}h(a)\overset{\mbox{\tiny by}\,\eqref{h}}{=\hskip -1mm=\hskip -1mm=\hskip -1mm=}f(a)\overset{\mbox{\tiny by}\,\eqref{vanishG}}{=\hskip -1mm=\hskip -1mm=\hskip -1mm=}0.
\end{split}
\end{equation*}
So $\ell$ vanishes on a dense open subset of $X$; whence $\ell=0$.
Thus, it is proved that
\begin{equation}\label{vani}
\textstyle\sum_{i=1}^m s_it_i^g=0\quad\mbox{for every}\;g\in G.
\end{equation}
%%From property (q${}^{\ }_2$) and \eqref{kXG}

   Since $Z$ is affine and $\rho^*(k(Z))=k(X)^G$,  the field of fractions of $\rho^*(k[Z])$ is $k(X)^G$.\;This implies that $t_1,\ldots, t_m$ are linearly independent over $k(X)^G$.\;In turn, by Artin's theorem
\cite[{\S}7, no.\,1,\,Thm.\,1]{Bou59}, this linear independency yields the existence of elements $g_1,\ldots, g_m\in G$
such that
%%$g^*_i(t_j)$
\begin{equation}\label{det}
{\rm det}\big(t_i^{g_j}\big)\neq 0.
\end{equation}

 Combining \eqref{vani} and \eqref{det} we obtain $s_1=\ldots=s_m=0$.\;From this, \eqref{st}, and \eqref{h}, we then infer that
$f=0$, as claimed.
 \end{proof}

\subsection*{9.\;Distinguished connected subgroups of \boldmath$
\AX$} Some collections $\mathcal I$ of unital algebraic families in $\AX$ are  naturally %%accentuated
distinguished.
They
 generate
%% accentuated
distinguished connected subgro\-ups
$\AX_{\mathcal I}$ of $\AX$ that are of interest.
%% for exploration.

The first example is
 the collection $\mathcal U$ of all unital algebraic families in $\AX$.
 We shall denote
$\AX_{\mathcal U}$
by $\AX^0$
and call it
%%it
the {\it identity component of $\AX$}\!.
The group $\AX/\AX^0$ will be called {\it the component group of $\AX$}.

\begin{proposition}\label{fini} Let $X$ be an irreducible variety such that $\AX$ is a finite group.\;Then
$\AX^0=\{{\rm id}^{\ }_X\}$.
\end{proposition}
\begin{proof} Let $\fat$ be a unital algebraic family in $\AX$.\;Take a point $x\in X$.\;Irreducibility of $T$ implies irreducibility of the image $I_x$ of morphism \eqref{psix}.\;Finite\-ness of $\AX$ (resp.,\,unitality of $\fat$) implies finiteness of $I_x$ (resp.,\,$x\in I_x$).\;This yields $I_x=\{x\}$, i.e.,\,$\varphi^{\ }_T=\{{\rm id_X}\}$; whence the claim.
 \end{proof}

\begin{remark} For any finite group $G$, there is a smooth affine irreducible variety $X$ such that $\AX$ and $G$ are isomorphic; see \cite{Jel94}.
\end{remark}

The component group of $\AX$, in contrast to that of an algebraic group, may be infinite.

\begin{remark} If $k$ is uncountable, then the same argument as in the proof of Proposition \ref{fini} shows
that if $\AX$ is countable (such $X$ do exist, see Examples \ref{count}, \ref{qua} below), then $\AX^0=\{{\rm id}^{\ }_X\}$ and hence the component group of $\AX$ is countable.

\begin{example}\label{count}
 Let $X$ be a surface in ${\bf A}\!^3$ defined by the equation $x^2_1+x^2_2+x^2_3=x^{\ }_1x^{\ }_2x^{\ }_3+a$ where %%$x_1, x_2, x_3$ are the standard coordinate functions on ${\bf A}\!^3$ and
 $a\in k$.\;By \cite{Hut74}, if $a$ is generic, then $\AX$ contains a subgroup of finite index which is a free product of three subgroups of order\;$2$.
\end{example}
\end{remark}

\begin{example}\label{qua}
Let ${\rm char}\,k=0$ and let $X$ be a smooth irreducible quartic in ${\bf P}^3$.\;Then $\AX^0=\{{\rm id}^{\ }_X\}$ by \cite{Mat63}, and, according to the classical Fano--Severi result, for a sufficiently general $X$ there is a bijection between $\AX$ and the (countable) set of solutions $(a, b), a>0$ of the Pell equation $x^2-7y^2=1$ (see \cite[pp.\,353--354]{MM64}).
\end{example}

\begin{example}\label{torus}
Let $X$ be
the underlying variety of an algebraic torus $G$ of dimension $n>0$.\;The automorphism group ${\rm Aut}_{\rm gr}(G)$ of the algebraic group $G$
is embedded in $\AX$ and isomorphic to ${\rm GL}_n({\mathbf Z})$.\;The map
$G\to\AX$, $g\mapsto \ell_g$, where $\ell_g\colon X\to X$, $x\mapsto gx$, identifies $G$ with a subgroup of $\AX$.\;By
%%These
%%action of $G$ on $X$ by
%%left translations embeds $G$ in $\AX$ too.\;These
%%two subgroups generate $\AX$; more
%%precisely, by
\cite[Thm.\,3]{Ros61},
\begin{equation}\label{sdp}
\AX = {\rm Aut}_{\rm gr}(G)\ltimes G.
\end{equation}

Let $\fat$ be a unital algebraic family in $\AX$.\;By
\cite[Thm.\,2]{Ros61}  there are the morphisms $\alpha\colon T\to G$ and
$\beta\colon X\to X$ such that
%%(see \eqref{wtilde})
$\varphi_t(x)
%%\overset{\mbox{\tiny see}\,\eqref{wtilde}}{=\hskip -1mm=\hskip -1mm=\hskip -1mm=}
=\widetilde\varphi (t, x)=\ell_{\alpha(t)}(\beta(x))$ for every $t\in T$, $x\in X$ (see \eqref{wtilde}).\;Put
%%(see \eqref{wtilde}).\;Put
$s:=\beta(e)$.\;Since $(\ell_{s^{-1}}\circ \beta)(e)=e$,
\cite[Thm.\,3]{Ros61} implies that $g:=\ell_{s^{-1}}\circ \beta\in {\rm Aut}_{\rm gr}(G)$.\;From $\beta=\ell_s\circ g$ we then infer that $\varphi_t(x)=\ell_{\alpha(t)}(\ell_{s}(g(x))=\ell_{\alpha(t)s}(g(x))$;\;whence $\varphi^{\ }_t=\ell_{\alpha(t)s}\circ g$.
%%
%%
%%\cite[Thms.\,2 and 3]{Ros61} that there exist a morphism
%%$\alpha\colon T\to G$ and the elements $s\in G$, $g\in {\rm Aut}_{\rm gr}(G)$ such that
%%$\widetilde\varphi (t, x)=\ell_{\alpha(t)s}(g(x))$ for every $t\in T$, $x\in X$ (see \eqref{wtilde}), i.e., %%$\varphi^{\ }_t=\ell_{\alpha(t)s}\circ g$.\;
This, \eqref{sdp}, and the unitality of $\fat$ imply that $g=\{{\rm id}^{\ }_X\}$. Therefore, $\varphi^{\ }_T\subseteq G$.\;This proves that $\AX^0=G$ and the component group of $\AX$ is isomorphic to ${\rm GL}_n({\mathbf Z})$.
\end{example}

\begin{example} By \cite[Cor.\,1]{Ram64}, $\AX^0$ is a connected algebraic group if $X$ is an irreducible complete variety (and, in fact, more generally, semi-complete variety, i.e., if
for any torsion free coherent algebraic sheaf $\mathcal F$ on $X$, the $k$-vector
space $H^0(X, \mathcal F)$  of sections is finite dimensional).
\end{example}

\begin{theorem}
%%The affine Cremona group
${\rm Aut}({\bf A}\!^n)={\rm Aut}({\bf A}\!^n)^0$
%%is connected
for every $n$.
%%$n\leqslant 2$.
\end{theorem}
\begin{proof} We shall apply the argument going back to \cite{Ale23} and utilized in \cite[Lemma 4]{Sha82}.\;Let $x_1,\ldots, x_n$ be the standard coordinate functions on ${\bf A}\!^n$.\;Any element of
$\AX$ is a composition of an element of the affine group ${\rm Aff}_n:=\{a\in \AX\mid {\rm deg}\,a^*(x_i)\leqslant 1\; \mbox{for every}\; i\}$ and an element $g\in\AX$ such that
\begin{equation}\label{forms1}
g^*(x_i)=x_i+{\textstyle\sum_{j=2}^{d}}f_{ij},\quad i=1,\ldots, n,
\end{equation}
where every $f_{ij}\in k[x_1,\ldots, x_n]$ is either zero or a form of degree $j$.\;Given that ${\rm Aff}_n$ is a connected algebraic group, this reduces the proof to demonstrating that $g$ is contained in a unital algebraic family in $\AX$.

This can be done as follows.\;For every $t\!\in\! k\!=\!{\mathbf A}^1$, $t\!\neq\! 0$, define $h_t\!\in\!  {\rm Aff}_n$\;by
\begin{equation}\label{ho}
h_t^*(x_i)=tx_i,\quad i=1,\ldots, n,
\end{equation}
and put $g_t:=h_t^{-1}\circ g\circ h_t\in\AX$.\;Then \eqref{forms1} and \eqref{ho} yield
\begin{equation}\label{forms2}
g_t^*(x_i)=x_i+{\textstyle\sum_{j=2}^{d}}t^{j-1}f_{ij},\quad i=1,\ldots, n.
\end{equation}
Putting $g_0:={\rm id}_{{\mathbf A}_{\ }^{\!n}}$, we deduce from  \eqref{forms2}
that $\{g_t\}_{t\in {\bf A}\!^1}$ is
a unital algebraic family
in $\AX$, and from \eqref{forms1} that $g_1=g$.\;This completes the proof.
 \end{proof}

A series of examples is obtained taking $\mathcal I$ to be a part of
the collection $\mathcal G$ of all algebraic families $\fat$ such that $T$ is a connected algebraic group and $\widetilde \varphi$ defined by \eqref{wtilde} is
an action of $T$ on $X$.\;In this case, $\AX_{\mathcal I}$ is a subgroup of $\AX$ generated, as an abstract group, by a  collection of some connected algebraic subgroups of $\AX$.\;For
%%Provided that
${\rm char}\,k=0$, the subgroups $\AX_{\mathcal I}$ of this type were studied in \cite[Sect.\,1]{AFKKZ} where they are called ``algebraically generated groups of automorphisms''.\;Propositions 1.3, 1.5 and Theorem 1.13 of \cite{AFKKZ} are the special cases  of, respectively, the above Lemma \ref{orbit}, Theorem \ref{main1}, and Theo\-rem\;\ref{main2}.

Some interesting parts $\mathcal I$ of $\mathcal G$ are obtained as collections of all families $\fat$ in $\mathcal G$ such that the algebraic group $T$ has a certain property.

For instance, requiring that $T$ is affine one obtains the collection ${\mathcal G}_{\rm aff}$.
Theorems \ref{appl1} and \ref{appl2} give examples of dependency between the groups $\AX_{\mathcal G}$, $\AX_{{\mathcal G}_{\rm aff}}$ and geometric properties of $X$.\;Here is another example.

\begin{example}
If $\AX_{{\mathcal G}_{\rm aff}}\neq\{\rm id^{\ }_X\}$, then $X$ is birationally isomorphic to the product of ${\bf A}\!^{1}$ and a variety of dimension $\dim X-1$.\;This follows from
\cite[Cor.\,1]{Mat63}.
%%\;In particular, if $X$ is an irreducible smooth %%complete variety that has a multicanonical system %%$|nK|$ for some $n>0$, then
%%$\AX_{{\mathcal G}_{\rm aff}}=\{\rm id^{\ }_X\}$.
\end{example}

Developing the idea of \cite[Def.\,1.36]{Pop11}, one obtains another example of an
interesting collection of families %%by
%%fixing a connected algebraic group $F$ and
taking $\mathcal I$ to be
the collection $\mathcal G(F)$ of all families $\fat$ in $\mathcal G$ such that
$T$ is isomorphic to a fixed connected algebraic group $F$.

For $F\!=\!{\bf G}_{\rm a}$  this yields
the important subgroup $\AX_{\mathcal G({\bf G}_{\rm a})}$ in $\AX$,
%%(at the irrelevant assumption
%%$X={\bf A}\!^n$)
in\-troduced\footnote{At the irrelevant assumption
$X={\bf A}\!^n$.
}
in \cite[Def.\,2.1]{Pop05}
and called in this paper ``$\partial$-generated subgroup''\!.\;Its close relation
to constructing a big stock of varieties with trivial Ma\-kar-Limanov invariant was shown
%%demonstrated 
in \cite{Pop11}.\;Later in \cite{AFKKZ}  the transitivity pro\-perties of
%%this subgroup
$\AX_{\mathcal G({\bf G}_{\rm a})}$ (called
in this paper ``the special automorphism group of $X$'' and denoted by\footnote{A hardly felicitous notation, as in the literature ${\rm SAut}$ denotes entirely different concept\,---\,the group
of semialgebraic automorphisms, see Y.\;Z.\;Flicker, C.\;Scheiderer, R.\;Sujatha, {\it Grothendieck's theorem
on non-abelian $H^2$ and local-global principles}, J. Amer. Math. Soc.
{\bf 11} (1998), no. 3, 731--750.} ${\rm SAut}(X)$) were studied.\;By \cite[Lemma 1.1]{Pop11},
$\AX_{\mathcal G({\bf G}_{\rm a})}$ coincides with the subgroup of $\AX$ generated by all
connected affine subgroups of $\AX$ that have no nontrivial characters.

Another interesting case is $F={\bf G}_{\rm m}$.\;Since the union of all maximal tori of a connected reductive group is dense in it,\;$\AX_{\mathcal G({\bf G}_{\rm m})}$ coincides with the subgroup of $\AX$ generated by all
connected reductive subgroups of $\AX$.\;This implies that
$$\AX_{{\mathcal G}_{\rm aff}}=\AX_{{\mathcal G}({\bf G}_{\rm a})\bigcup{\mathcal G}({\bf G}_{\rm m})}.$$
Indeed, let $H$ be a connected affine algebraic group with a maximal torus $T$ and the unipotent radical ${R}_u(H)$, and let
$\pi\colon H\!\to\! H/{R}_u(H)$ be the ca\-no\-nical projection.\;By \cite[Prop.\,11.20]{Bor91}, $\pi(T)$ is a maximal torus in $H/{R}_u(H)$.\;The conjugacy of maximal tori  and the density of their union in $H/\!{R}_u(H)$ yield $H/{R}_u(H)=\pi(S)$ for the subgroup $S$ in $H$ generated by
all maximal tori.\;Whence the claim.

\subsection*{10.\;Proof of Theorem \ref{appl1}} Since $G$ lies in  $\AX_{{\mathcal G}_{\rm aff}}$, by Corollary \ref{c2} it suffices to show that neither of the $\AX_{{\mathcal G}_{\rm aff}}$-orbits is open in\;$X$.

Assume the contrary and let $\mathcal O$ be an $\AX_{{\mathcal G}_{\rm aff}}$-orbit open in $X$.\;Take a point $x\in \mathcal O$.\;By Theorem  \ref{main1}, a certain  family $\fat$ derived from ${\mathcal G}_{\rm aff}$ is exhaustive for the action  of $\AX_{{\mathcal G}_{\rm aff}}$ on $X$.\;Then
$\mathcal O$ is the image of morphism
%%$T\to X$, $t\mapsto \varphi_t(x)$.
\eqref{psix}.\;Since $\mathcal O$ is open in $X$, this morphism
is dominant.\;On the other hand, the definitions of
derived family and  ${\mathcal G}_{\rm aff}$ imply that $T$ is
a product of underlying varieties of connected affine algebraic groups.\;But such underlying
varieties are rational (see \cite[Lemma\,2]{Pop13} for a four-lines proof;\;we failed to find an earlier  reference for a proof valid in arbitrary characteristic).\;Hence
$T$ is a rational variety.\;This and the dominance of morphism \eqref{psix} then imply that $X$ is unirational\,---\,a contradiction.\quad  $\square$

\subsection*{11.\;Proof of Theorem \ref{appl2}} Let $X$ be nonunirational.\;Assume that $\AX$ contains a nontrivial connected affine algebraic subgroup $C$.\;Then there exists a point $x\in X_2$ such that $X_2\cap C(x)$ is an irreducible
locally closed set of positive dimension.\;Hence there exists a point $y\in X_2\cap C(x)$, $y\neq x$.\;By the condition of $2$-transitivity, for every point $z\in X_2$, $z\neq x$, there exists an element $g\in\AX$ such that $g(x)=x$, $g(z)=y$.\;This implies that for the subgroup $H:=g^{-1}\circ C\circ g$ we have
$z\in H(x)$.\;Therefore, for the connected subgroup $G$ of $\AX$ generated by
all conjugates of $C$ in $\AX$ we have $X_2\subseteq G(x)$; whence $G(x)$ is open in $X$.

From this, arguing as in the proof of Theorem  \ref{appl1}, we deduce that $X$ is unirational\,---\,a contradiction.\;Hence $\AX$ does not contain nontrivial connected affine algebraic subgroups.

Now assume that $\AX$ contains a nontrivial connected nonaffine algebraic subgroup $A$.\;Since, as we proved, there are no nontrivial
connected affine algebraic subgroups in $A$, the structure theorem on algebraic groups \cite{Bar55}, \cite{Ros56}
implies that $A$ is a nontrivial abelian variety.\;The same argument as above for $C$ then shows that
%%there is an orbit
%%one of the orbits
%%$\mathcal O$ of
the connected subgroup of $\AX$ generated by
all conjugates of $A$ in $\AX$ has an orbit $\mathcal O$  which is open in $X$ and admits
%%that there exists
a surjective morphism
$Z\to\mathcal O$, where $Z$ is a product of se\-veral copies of the underlying variety of $A$.\;Since $Z$ is a complete variety, this implies that $X$ is complete as well and $X=\mathcal O$.

The completeness of  $X$ implies that  $\AX^0$ is a connected algebraic group and $\AX/\AX^0$ is at most countable \cite{MO67}.\;Since $\AX^0$ does not contain nontrivial connected algebraic subgroups, the same argument as above yields that $\AX^0$ is a nontrivial abelian variety acting transitively on $X$.\;The commutativity of $\AX^0$ and the faithfulness of its action
%%of $\AX^0$
on $X$
then imply that the $\AX^0$-stabilizer of every point of $X$ is trivial.

Take a point $x\in X_2$ and let $\AX_x$ be its $\AX$-stabilizer.\;The assumption of generic $2$-transitivity of the action of $\AX$ on $X$ implies that there is an $\AX_x$-orbit containing $X_2\setminus \{x\}$.\;But this orbit is at most countable since $\AX_x\cap \AX^0=\{e\}$, while  $X_2\setminus \{x\}$, being open in $X$, is uncountable (e.g., because an affine open subset of $X_2\setminus \{x\}$ is a branched covering of an affine space by the Noether lemma)\,---\,a contradiction.\;This completes the proof.
\quad  $\square$

\subsection*{12.\;Proof of Corollary \ref{proje}} Assume that $X$ is nonunirational.\;Then by Theo\-rem\;\ref{appl2} the group
 $\AX$ contains no nontrivial connected algebraic subgroups.\;Since $X$ is complete, this implies that
 %%$\AX^0=\{e\}$ and at most countability of
 $\AX$ is at most countable.\;A contradiction with the assumption of generic $2$-transitivity of the action of $\AX$ on $X$ is then obtained using
 the same argument as in the end of the proof of Theorem \ref{appl2}.
\quad  $\square$

%%\eject 

\subsection*{13. Calogero--Moser spaces}%%
%%First, we give
%%{13.\;Proof of Corollary \ref{ColMos}}
\begin{proof}[Proof of Corollary {\rm \ref{ColMos}}]
According to \cite{Wil98}, ${\mathcal C}_n$ is an irreducible smooth af\-fine variety.\;By \cite[Thm.\;1]{BEE14}, the natural action of ${\rm Aut}({\mathcal C}_n)$ on ${\mathcal C}_n$ is $2$-transiti\-ve.\;There are nontrivial connected algebraic subgroups in ${\rm Aut}({\mathcal C}_n)$: for instance, the action of  ${\rm GL}_1$ on
%%$\widetilde{\mathcal C}_n$
$\{(A, B)\in {\rm Mat}_n({\mathbf C})^2\mid {\rm rk}([A, B]+I_n)=1\}$ given by $t\!\cdot\! (X,Y)\!:=\!(t^{-1}X, tY)$ descends to ${\mathcal C}_n$.\;Unirationality of  ${\mathcal C}_n$ then follows from Theorem\;\ref{appl2}.
 \end{proof}
%%\hfill $\square$

\begin{remark}\label{ration}  One can show that ${\mathcal C}_n$ is actually rational. The proof
is based on \cite[Prop.\;1.10]{Wil98} and goes as follows.\footnote{I thank Y.\;Berest who drew in \cite{Ber14} my attention to \cite[Prop.\;1.10]{Wil98} and gave
another argument, due to G. Wilson, that deduces from it the rationality of ${\mathcal C}_n$.}\;Endow ${\mathbf A}\!^n$ with the standard action of the symmetric group $S_n$ and consider the diagonal action of $S_n$ on ${\mathbf A}\!^n\times {\mathbf A}\!^n$.\;It follows from \cite[Prop.\;1.10]{Wil98} that
${\mathcal C}_n$ is birationally isomorphic to $({\mathbf A}\!^n\times {\mathbf A}\!^n)/S_n$.\;By the No-name Lemma (see, e.g., \cite[Lemma\;1]{Pop13}), $({\mathbf A}\!^n\times {\mathbf A}\!^n)/S_n$ is birationally
isomorphic to ${\mathbf A}\!^n\times ({\mathbf A}\!^n/S_n$). Since ${\mathbf A}\!^n/S_n$ is isomorphic to
${\mathbf A}\!^n$, the claim follows.
\end{remark}

\subsection*{14.\;Proof of Corollary \ref{reptyp}} Irreducibility of $Q_{m,n}(\tau)$ is proved in \cite[Thm.\;II.1.1]{LBP87}.\;By \cite[Thm.\;1.4]{Rei93}, for $m\geqslant 3$ the natural action of group ${\rm Aut}(Q_{m,n}(\tau))$ on  $Q_{m,n}(\tau)$ is $2$-transitive.
There are nontrivial connected algebraic subgroups in ${\rm Aut}(Q_{m,n}(\tau))$:\;for instance, the action of  ${\rm GL}_1$ on ${\rm Mat}_n(k)^m$ by scalar multiplication
%%$t\cdot (X_1,\ldots, X_m):= (tX_1,\ldots, tX_m)$
induces a representation type preser\-ving action on ${\rm Mat}_n(k)^m/\!\!/{\rm PGL}_n(k)$; see\;\cite[Prop.\;4.1]{Rei93}.\;The uniratio\-nality
%%and the obtained action preserves
%%$Q_{m,n}(\tau)$ and is nontrivial on it, see .
 of  $Q_{m,n}(\tau)$ then follows from Theorem\;\ref{appl2}.
\quad  $\square$

\subsection*{15.\;Appendix} Here is the folklore statement mentioned in the introduction:
%%Section 1:

%%\begin{thmnonumber}
 \begin{theorem}\label{folk}
 Let ${\rm char}\,k=0$. If $\,X$ is an irreducible affine variety, $\dim X\geqslant 2$, and $\AX$ contains a one-dimensional algebra\-ic unipotent subgroup $U$, then the group $\AX$ is infinite dimensional.
\end{theorem}%%{thmnonumber}

\begin{proof} By Rosenlicht's theorem \cite{Ros56}, ${\rm tr\,deg}_kk(X)^U\!\geqslant\! \dim X\!-\!\dim U\!>\!0$. By  \cite[Thm.\;3.3]{PV94}, the unipotency of $U$ yields the equality
  ${\rm tr\,deg}_kk[X]^U\hskip -1mm={\rm tr\,deg}_kk(X)^U$. Since $X$ is affine,
   $U=%%U_\partial:=
   \{{\rm exp}\,t\partial\mid t\in k\}$, where  $\partial$ is a local\-ly nilpotent de\-ri\-vation of $k[X]$.\;The claim then follows from the inclusion
   %%$\{{\rm exp}\,th\partial\mid t\in k\}$, where
   $\{{\rm exp}\,{h\partial}\mid h\in k[X]^U\}\subseteq \AX$.
    \end{proof}

   \begin{remark} Example \ref{torus} shows that ``unipotent'' in Theorem \ref{folk}  cannot be
   dropped.
   %%replaced
   %%by ``semisimple''.
   \end{remark}

 \end{document}